\documentclass{article}
\usepackage{amsmath,amsthm}
\usepackage{amssymb}
\usepackage{times}
\usepackage{xspace}
\usepackage{xcolor}
\usepackage{graphics}
\definecolor{darkblue}{rgb}{0,0,.7}
\usepackage[breaklinks,colorlinks,urlcolor=darkblue,linkcolor=darkblue,citecolor=darkblue,implicit=false]{hyperref}

\newtheorem{theorem}{Theorem}

\newtheorem{proposition}[theorem]{Proposition}
\newtheorem{lemma}[theorem]{Lemma}
\newtheorem{corollary}[theorem]{Corollary}
\theoremstyle{definition}
\newtheorem{definition}[theorem]{Definition}
\newtheorem{remark}[theorem]{Remark}
\newtheorem{example}[theorem]{Example}

\newcommand{\N}{\mathbb{N}}

\title{Sparse variational regularization with oversmoothing penalty term in the scale of sequence spaces}

\author{Robert Plato\footnotemark[2] \and Bernd Hofmann\footnotemark[3]}

\newcommand{\mytau}{r}
\newcommand{\myr}{a}
\newcommand{\uhat}{\widehat{u}}
\newcommand{\Rppur}[1][p]{\mathcal{R}_{#1}}
\newcommand{\Rp}[2][p]{\mathcal{R}_{#1}(#2)}
\newcommand{\fdelta}{v^\delta}
\newcommand{\mygama}{\gamma_1}
\newcommand{\mygamb}{\gamma_2}

\newcommand{\rp}[1]{#1}
\newcommand{\rpb}[1]{#1}
\newcommand{\bh}[1]{#1}
\newcommand{\searchedsol}{sought-after solution\xspace}
\newcommand{\searchedsolb}{desired solution\xspace}

\newcommand{\myprime}{*}

\begin{document}
\maketitle

\renewcommand{\thefootnote}{\fnsymbol{footnote}}

\footnotetext[2]{Department of Mathematics, University of Siegen,
Walter-Flex-Str.~3, 57068 Siegen, Germany.}
\footnotetext[3]{Chemnitz University of Technology, Faculty of Mathematics,  09107 Chemnitz, Germany.}

\newcounter{enumcount}
\newcounter{enumcountroman}
\renewcommand{\theenumcount}{(\alph{enumcount})}
\bibliographystyle{plain}
\begin{abstract}
In this work, we consider a class of linear ill-posed problems with operators that map from the sequence space $ \ell_\mytau $ ($\mytau \ge  1 $) into a Banach space and in addition satisfy a conditional stability estimate in the scale of sequence spaces $ \ell_q, \, q \ge 0 $. For the regularization of such problems in the presence of deterministic noise, we consider variational regularization with a penalty  functional either of the form $ \Rppur =\Vert \cdot \Vert_p^p $ for some $ p > 0 $ or in form of  the counting measure $ \Rppur[0]=\Vert \cdot \Vert_0 $.
The latter case guarantees sparsity of the corresponding regularized solutions. In this framework, we present first stability and then convergence rates for suitable a priori parameter choices. The results cover the oversmoothing situation, where the \searchedsolb does not belong to the domain of definition of the considered penalty functional. The analysis of the oversmoothing case utilizes auxiliary elements that are defined by means of hard thresholding. Such technique can also be used for post processing to guarantee sparsity.
\end{abstract}

\vspace{3mm}
\noindent
{\normalsize
\textbf{In memory of our esteemed advisor and distinguished colleague A. K. Louis}
}
\section{Introduction}
\label{sec:intro}

In this paper, we consider linear ill-posed problems in the scale of real sequence spaces $\{\ell_s=\ell_s(\mathbb{N})\}_{0 \le s <\infty}$. As a preparation, we recall some basic details about those sequence spaces. In fact, we have
\begin{align*}
\ell_s = \{ u = (u_k)_{k=1,2,\ldots} \mid
\Vert u \Vert_s < \infty \},
\quad
\Vert u \Vert_s := \big(\sum_{k=1}^\infty \vert u_k \vert^s\big)^{1/s},
\qquad 0 <  s < \infty.
\end{align*}
For $ s \ge 1 $, this defines a Banach space, and for $ 0 < s < 1 $ it gives a quasi-normed space, where a generalized triangle inequality holds. We also consider the case $ s = 0 $:
\begin{align*}
\ell_0 = \{ u = (u_k)_{k=1,2,\ldots} \mid \Vert u \Vert_0 < \infty \},
\quad \Vert u \Vert_0 := \# \{ k \in \mathbb{N} \mid u_k \neq 0 \},
\end{align*}
where $ \# \mathcal{A} $ denotes the number of elements of a set $ \mathcal{A} $.
Thus, $ u \in \ell_0 $ means
$ u_n \neq 0 $ for finitely many $ k $ only, i.e.~$ u $ is sparse.
For further reference we note that
\begin{equation} \label{eq:monoton}
\|u\|_t \le \|u\|_s \quad (0<s \le t \rp{\ < \ } \infty,\;\; u \in \ell_s),
\end{equation}
and
\begin{align*}
\lim_{p\to 0} \Vert u \Vert_p^p = \Vert u \Vert_0 \quad (u \in \ell_0).
\end{align*}
Below, we also make use of the interpolation inequality for sequence spaces.
\rp{As a preparation, we introduce the following functionals:
\begin{align}
\label{eq:rpdef}
\Rp{u}=\|u\|_p^p \quad (p>0), \qquad \Rp[0]{u}= \|u\|_0.
\end{align}
The interpolation inequality for sequence spaces then reads as follows:
for parameters $ 0 \le p < r < a < \infty $ and
$ \theta = \frac{\myr-r}{\myr-p} $ we have}
\begin{align}
\label{eq:interpol_ineq}
\Vert u \Vert_\mytau^\mytau
\le
\Rp{u}^\theta
\Vert u \Vert_\myr^{\myr(1-\theta)}
\quad (u \in \ell_p).
\end{align}
%
%
%
%
This is well known for $ p > 0 $, cf.~e.g.,
\cite[Proposition 6.10]{Folland[99]} or \cite[Lemma 9]{Tao[09]},
and in fact follows easily by a suitable application of H\"older's inequality.
It can then easily be extended to the case $ p = 0 $
by letting $ p \to 0 $ in (\ref{eq:interpol_ineq}).
%
%
%

In this study, we also make use of the following Radon--Riesz property of $ \ell_p, \, 0 < p < \infty $, also known as Kadec--Klee property. Note that
 in the proposition below, a componentwise convergence formulation is utilized instead of a weak or weak* convergence formulation. Our modification avoids usage of non-standard theory of dual spaces of non-locally convex topological vector spaces $ \ell_p, \ 0 < p < 1 $.
Note moreover that for bounded sequences in $ \ell_p $, with $ p > 1 $, componentwise convergence is equivalent with weak convergence, cf.~e.g., \cite[Theorem 8.20]{Nair[08]}.
In $ \ell_1  = c_0^\myprime $, a similar statement holds for weak* convergence.
%
\begin{proposition}[Radon--Riesz property] Consider for $ 0 < p < \infty $ the sequence $ \{u_{n}\}_{n=1}^\infty \subset \ell_p\, $ as well as the element $ u \in \ell_p $,
and let the following two conditions be satisfied:
(a) convergence of $ u_n $ to $ u $ as $ n \to \infty $ holds componentwise, i.e.~$ u_{n,k} \to u_k $ as $ n \to \infty $ for each $ k \ge 1 $,
and (b)
$ \Vert u_{n} \Vert_p \to \Vert u \Vert_p $ as $ n \to \infty $.
Then $ \Vert u_{n} - u \Vert_p \to 0 $ as $ n \to \infty $.
%
\label{th:radon_riesz}
\end{proposition}
\begin{proof}
For $ p >1 $, the statement of the proposition follows from the Radon--Riesz property of uniformly convex spaces, cf.~e.g., \cite[Proposition 3.32]{Brezis[11]}.
For $ p \le 1 $, the statement can be verified using elementary means from calculus, since it is formulated for $ \ell_p $ spaces here. The proof (which in fact works out for any $ 0 < p < \infty $) is left as an exercise; cf.~also \cite[Proposition 3.6]{Grasmair09b} for a similar proof technique, and also for a result closely related to our proposition in fact.
\end{proof}

In what follows, we frequently make use of indices
\begin{align}
0 \le p \le q \le \mytau < \infty, \quad 1 \le \mytau < \myr < \infty,
\label{eq:parameters}
\end{align}
so in particular the following chain of inclusions holds:
\begin{align*}
\ell_p \subset \ell_q \subset \ell_\mytau \subset \ell_\myr.
\end{align*}
%
The meaning of the parameters in (\ref{eq:parameters}) is as follows:
\begin{itemize}
\item The parameter $ p $ determines the stabilizing functional utilized in Tikhonov regularization introduced in (\ref{eq:Tik}) below,

\item the parameter $ q $ corresponds to the summability property of the entries
of the \searchedsolb $ u^\dagger $ of equation (\ref{eq:opeq}) considered below,
i.e.~we assume $ u^\dagger \in \ell_q $,
\item the parameter $ r $ determines the norm utilized for measuring the error, i.e.~we use $ \Vert \cdot \Vert_r $ here,

\item and the parameter $ a $  determines the conditional stability of the linear operator under consideration, cf.~(\ref{eq:twosided}) below.
\end{itemize}

The present article intends to be a contribution to the theory of variational regularization for linear ill-posed problems in abstract spaces, and we refer in this context for example to
the textbooks and monographs \cite{EHN96,Louis89,Scherzerbuch09,Schusterbuch12}.
We consider here, by using variational regularization, stable approximate solutions to a specific class of ill-posed linear operator equations, which attain the form
 \begin{equation} \label{eq:opeq}
A u = v \quad (u \in \ell_\mytau,\;v \in  V),
\end{equation}
where $V$ is a real Banach space and $A: \ell_\mytau \to V$ an injective bounded non-compact operator with non-closed range $\mathcal{R}(A)$,
for which the two-sided estimate
 \begin{equation} \label{eq:twosided}
d_1\,\|u\|_\myr \le \|A u\|_V \le d_2\,\|u\|_\myr  \qquad (u \in \ell_\mytau)
\end{equation}
is valid for constants $0<d_1 \le d_2 < \infty$. Evidently, \eqref{eq:twosided} shows that there is a bounded and continuously invertible linear operator $B: \ell_\myr \to V$, which is an extension of $A$ to $\ell_\myr$, and vice versa $A$ is the restriction of $B$ to $\ell_\mytau$.
Note that one can write $A=B \circ \mathcal{E}_\mytau^\myr: \ell_\mytau \to V$ with the embedding operator $\mathcal{E}_\mytau^\myr: \ell_\mytau \to \ell_\myr$.
\footnote{In the limiting case $r=a$, where the operator equation \eqref{eq:opeq} is well-posed, we have $A=B$ with a closed range $\mathcal{R}(A)$.}

Let
\begin{equation}
\label{eq:vdelta}
v^\delta \in V, \quad
\Vert v^\delta - v \Vert_V \le \delta, \quad \delta > 0,
\end{equation}
denote small perturbations of the right-hand side of the given equation
(\ref{eq:opeq}).
Variational regularization aims here at minimizing the Tikhonov functional
\begin{equation}
\label{eq:Tik}
T_\alpha^\delta(u):=\|Au-v^\delta\|_V^\sigma + \alpha\,\Rp{u}
\end{equation}
over $\ell_\mytau \;(1 \le \mytau < \myr)$, with $0 \le p \le \mytau$,
%
%
%
with exponents $\sigma>0$ and regularization parameters $\alpha>0$. Let denote by
$$ u_\alpha^\delta = \textup{arg} \min \,\{ T_{\alpha}^\delta(u) \,|\; u \in \ell_p \} $$
the associated regularized solution and mention that the non-negative functional $T_\alpha^\delta(u)$ attains finite values if and only if $u \in \ell_p$ holds.

\bh{We recall that variational approaches of type \eqref{eq:Tik} with penalty functionals $\mathcal{R}_p$ from \eqref{eq:rpdef} and $0 \le p \le 2$  are referred to in the literature as \emph{regularization with sparsity constraints} or \emph{sparse regularization}
(see, e.g.,~\cite{Bredies08,Burger13,Grasmair09,Jin12,Lorenz08,Ramlau08,Wei19}), where in particular it can be shown under weak additional assumptions that the regularized solutions $ u_\alpha^\delta$ are indeed sparse for all $0 \le p \le 1$.}

In the following lemma, restrictions of the operator $ A : \ell_\mytau \to V $ to     spaces $ \ell_s \ (1 \le s \le \mytau) $ are denoted by $ A $ again. The same holds for uniquely determined extensions of $ A $ to spaces $ \ell_s \ (\mytau \le s \le \myr $).
%
\begin{lemma} \label{lem:basics}
Under the assumptions stated above, we have
\begin{itemize}
\item[(a)] The linear operator $A$ is norm-to-norm continuous from $\ell_s$ to $V$ for $1 \le s \le \myr$, i.e.~there are positive constants $K_s$ such that $\|Au\|_V \le K_s\,\|u\|_s\;\forall u \in \ell_s$.
\item[(b)] For $ 1 \le s \le \myr$, the operator $A: \ell_s\to V$ is weak-to-weak continuous, i.e.~a weakly convergent sequence $u_n \rightharpoonup u_0$ in $\ell_s$ implies weak convergence $A u_n \rightharpoonup A u_0$ in $V$.
\item[(c)] For the non-reflexive space $\ell_1$ with the predual space $c_0$, the operator $A: \ell_1\to V$ is weak$^*$-to-weak continuous, i.e.~a weak$^*$-convergent sequence $u_n \rightharpoonup^* u_0$ in $\ell_1$ implies weak convergence $A u_n \rightharpoonup A u_0$ in $V$.
\end{itemize}
\end{lemma}
\begin{proof}Item (a) is a consequence of the right inequality in \eqref{eq:twosided} in combination with \eqref{eq:monoton}. Since every norm-to-norm continuous linear operator mapping between Banach spaces is also weak-to-weak continuous, this yields item (b).
\rpb{Finally, item (c) holds for $A=B \circ \mathcal{E}_1^\myr: \ell_1 \to V$ with the bounded linear operator $B:\ell_a \to V$ and the embedding operator $\mathcal{E}_1^\myr: \ell_1 \to \ell_\myr$,
because $a>1$. Namely, just for $a>1$, the embedding operator $\mathcal{E}_1^\myr$ is weak$^*$-to-weak continuous due to \cite[Lemma~9.5]{Flemmingbuch18}, because $\mathcal{E}_1^\myr e^{(k)}= e^{(k)}$
converges for $k \to \infty$ weakly to zero in $\ell_a$, where $e^{(k)}$ denotes the $k$-th unit sequence. Since the bounded linear operator $B$ is also weak-to-weak continuous,
the composite operator $A$ is  weak$^*$-to-weak continuous as required.} This completes the proof. \end{proof}
\subsection{Outline}
The paper is organized as follows: Section~\ref{sec:illposendnesstypes} \rpb{distinguishes in the sense of Nashed \cite{Nashed86} some ill-posedness types that may occur for bounded linear ill-posed problems in Banach spaces. We consider the case of \emph{injective} operators, since the focus is on an application to sequence spaces, where such operators occur. For extended studies, including the case of non-injective operators in Banach spaces, we refer to the recent paper \cite{HofKin25}}.
Section \ref{sec:wellposedness} is devoted to the well-posedness of Tikhonov regularization with oversmoothing penalties in $ \ell_p $ spaces.
In Section \ref{apriori_parameter_choice}, we introduce an appropriate a priori parameter choice for the variational regularization
(\ref{eq:Tik}) introduced above. In addition, corresponding convergence rates are presented, and a post processing process is considered that yields sparse regularizers without losing any accuracy. Finally, Section~\ref{conclusions_outlook} presents a conclusion and an outlook.
We note that the focus of Section \ref{apriori_parameter_choice} is on convergence analysis. In contrast, computational complexity issues are not considered in this paper. The same applies to the sparsity-promoting features of variational regularization utilizing penalty functionals
$ \Rppur $ for $ 0 < p \le 1 $.
%
\section{Ill-posedness types of linear operator equations in Banach spaces and a specific model scheme} \label{sec:illposendnesstypes}

In this section, we are going to apply the results of \cite[Section~4]{FHV15}, and repeat Figure~1 ibid for illustration, in order to classify our present discussion of oversmoothing regularization in sequence spaces from the perspective of Nashed's characterization of ill-posedness types introduced in the seminal paper \cite{Nashed86}.
\emph{Strictly singular and non-compact} injective bounded linear operators acting between Banach spaces, will play the most prominent role in these discussions, because such linear operators with non-closed range cannot occur as mappings between Hilbert spaces.

\subsection{General case distinctions} \label{subsec:general}

In the sense of \cite{Nashed86} we have the following definition:

\begin{definition} \label{def:type}
Let $T: X \to Y$ be an injective and bounded linear operator mapping between the infinite dimensional Banach spaces $X$ and $Y$.  Then the operator equation
\begin{equation} \label{eq:Nashed}
T x=y \quad (x \in X,\;y \in  Y)
\end{equation}
is called {\sl well-posed} if the range $\mathcal{R}(T)$ of $T$ is a closed subset of $Y$. Consequently \eqref{eq:Nashed} is called {\sl ill-posed} if the range $\mathcal{R}(T)$ is not closed, i.e.~$\mathcal{R}(T) \not= \overline{\mathcal{R}(T)}^{Y}$.
In the ill-posed case, the equation (\ref{eq:Nashed}) is called {\sl ill-posed of type I} if the range $\mathcal{R}(T)$ contains an {\sl infinite dimensional closed subspace}, and it is called {\sl ill-posed of type~II} otherwise.
\end{definition}

The diagram in Figure~1 
below illustrates the different cases occurring for equation \eqref{eq:Nashed} in our Banach space setting. We precede an explanation with the following definition.

\begin{definition} \label{def:strictlysingular}
A bounded linear operator $T:X \to Y$ mapping between the infinite dimensional Banach spaces $X$ and $Y$ is \emph{strictly singular} if its restriction to an infinite dimensional closed subspace of $X$ is never an isomorphism.
\end{definition}

By definition, the operator $T:X \to Y$ is strictly singular if and only if the range $\mathcal{R}(T)$ of $T$ does not contain  an infinite dimensional closed subspace of $Y$. Consequently, this property of strict singularity separates the left part of Figure~1 
from the right.
Not strictly singular operators $T$ (left part of Figure~1) 
lead either to \emph{well-posed} operator equations \eqref{eq:Nashed} if $T$ is \emph{continuously invertible} or to equations \emph{ill-posed of type~I} if $T$ has a non-closed range $\mathcal{R}(T)$. However, bounded injective and strictly singular linear operators $T$
mapping between infinite dimensional Banach spaces (right part of Figure~1) 
have always a non-closed range $\mathcal{R}(T)$ and lead therefore to operator equations \eqref{eq:Nashed} \emph{ill-posed of type~II} (see, e.g., \cite[Prop.~4.6]{FHV15}). The \emph{compact} operators $T$ form a significant subclass
for this case, which in particular fills out the right part of the figure completely in a Hilbert space setting. Namely, a bounded injective operator $T$ mapping between infinite dimensional Hilbert spaces is strictly singular if and only if it is compact. In Banach spaces the situation is more complex, because injective bounded non-compact and strictly singular linear operators occur, which show that the compact operators need not fill out the right part of the figure. Embedding operators  $\mathcal{E}_\mytau^\myr: \ell_\mytau \to \ell_\myr\;(1 \le r <a)$ mapping between sequence spaces with different indices deliver suitable counterexamples.

\begin{figure}[ht]\label{fig:FHV15}
\centerline{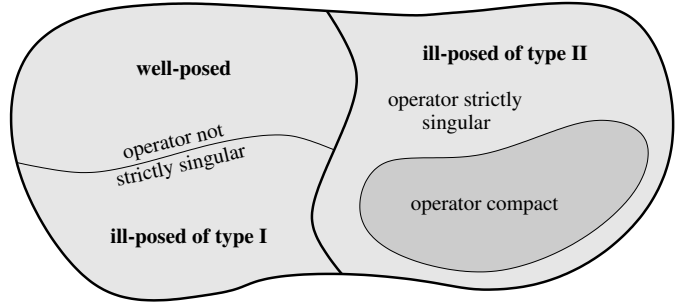}
\caption{Relations between strict singularity,  compactness and type of ill-posedness for equations in Banach spaces with injective bounded linear operators (cf.~\cite[p.~287]{FHV15}).}
\end{figure}

\subsection{A specific model scheme adopted to sequence spaces} \label{subsec:scheme}
\label{spec_model}%
For generating the model scheme, we consider infinite dimensional real Banach spaces $U,V,W$ and injective bounded linear operators $A: U \to V$, $S: U \to W$ as well as $B: W \to V$. In particular, let $B$ be a \emph{continuously invertible} operator, $S$ a \emph{non-compact and strictly singular} operator and $A= B \circ S$ a composition operator, which is then evidently also non-compact and strictly singular. We consider the couple of linear operator equations
 \begin{equation} \label{eq:A}
A u = v \quad (u \in U,\;v \in  V)
\end{equation}
and
\begin{equation} \label{eq:B}
B w= v \quad (w \in W,\;v \in V),
\end{equation}
where the latter equation \eqref{eq:B} is \emph{well-posed}, whereas the former equation \eqref{eq:A} is \emph{ill-posed of type~II} with \emph{non-compact and strictly singular} forward operator $A$.

We adopt this scheme to a situation introduced in Section~\ref{sec:intro} for the  \emph{scale of sequences spaces} introduced above. Let in this context $1 \le \mytau <\myr<\infty$. Moreover, consider $U:=\ell_\mytau$, $W:=\ell_a$ and $V$ as an arbitrary real Banach space. In particular, we define $S:=\mathcal{E}_\mytau^\myr: \ell_\mytau \to \ell_ \myr$ as the \emph{embedding} from $\ell_\mytau$ to $\ell_\myr$, which is an \emph{injective, bounded, non-compact and strictly singular} linear operator (see, e.g., \cite{GoldThorp63,Lef14}). Then for any prescribed continuously invertible operator $B: \ell_\myr \to V$,  the operator equation \eqref{eq:A}, which is ill-posed of type~II, attains the form \eqref{eq:opeq} with the bounded injective, non-compact and strictly singular linear operator $A=B \circ \mathcal{E}_\mytau^\myr: \ell_\mytau \to V$.

\section{Well-posedness} \label{sec:wellposedness}

In this section, we discuss the \emph{existence and stability} of minimizers $u_\alpha^\delta$ to the Tikhonov functional \eqref{eq:Tik}, where we refer in this context to the blueprints from \cite{Flemmingbuch18,HKPS07} and \cite{Scherzerbuch09,Schusterbuch12}. In particular, we distinguish
for the sketch of a proof the three cases as there are \linebreak
$(i):$ $\mytau>1$ with $0 < p \le \mytau$\,; $\;(ii):$ $0<p \le 1=\mytau$\,;\, and \, $(iii):$ $p=0$ of occurring penalty functionals $\mathcal{R}$.

\begin{proposition} \label{pro:well}
Let $0 \le p \le \mytau$ be fixed. Then there exists for all $\alpha>0$ and $v^\delta \in V$ a regularized solution $u_\alpha^\delta \in \ell_p$ minimizing the Tikhonov functional \eqref{eq:Tik} over $\ell_\mytau$, i.e.~$T_\alpha^\delta(u_\alpha^\delta)=\inf_{u \in \ell_\mytau} T_\alpha^\delta(u)$.
Moreover, every minimizing sequence  $\{u_n\}_{n=1}^\infty \subset \ell_p$ such that $\lim_{n \to \infty} T_\alpha^\delta(u_n)=T_\alpha^\delta(u_\alpha^\delta)$ has a subsequence $\{u_{n_k}\}_{k=1}^\infty$ that converges strongly in $\ell_\mytau$ as $k \to \infty$
to a minimizer of the Tikhonov functional $T_\alpha^\delta$.
\end{proposition}
\begin{remark}\rm
As a consequence Proposition~\ref{pro:well} the regularized solutions $u_\alpha^\delta$ are stable (in the respective weak or weak$^*$ sense) with respect to small perturbations in the norm of $V$ for the data $v^\delta \in V$.
\end{remark}

\begin{proof}[Sketch of a proof of Proposition~\ref{pro:well}]
To prove the proposition we follow the lines of Section~4.1.1 from \cite{Schusterbuch12}. In particular, we have to prove with respect to the basis space $\ell_\mytau$ some \emph{stabilizing property} of the penalty functional $\Rp{u}=\|u\|_p^p$ for $p>0$ and the \emph{ weak (or weak$^*$) lower semi-continuity} of the Tikhonov functional $T_\alpha^\delta$. For $p=0$, the discussion is a bit more complex.
\begin{itemize}
\item[(i)] For $\mytau>1,\;0 < p \le \mytau$, the stabilizing property means that the sublevel sets $L_c:=\{u \in \ell_p\,|\;\|u\|_p^p \le c\}$ are bounded and therefore relatively weakly compact subsets of $\ell_\mytau$ for all constants $c>0$. Because of \eqref{eq:monoton} and $p \le \mytau$ implying that
$L_c \subset \{u \in \ell_\mytau\,|\;\|u\|_\mytau^p \le c\}$, this is the case.
Moreover, the penalty functional $\Rppur$ is here weakly lower semi-continuous with respect to $\ell_\mytau$. This is also the case for the misfit functional $\|A \cdot-v^\delta\|_V^\sigma$ as a consequence of item~(b) from Lemma~\ref{lem:basics} and because norm functionals are
always weakly lower semi-continuous. Then the Tikhonov functional $T_\alpha^\delta$ is everywhere weakly lower semi-continuous with respect to $\ell_\mytau$. Then minimizers $u_\alpha^\delta$ always exist and are stable.
In addition, for each minimizing sequence $ \{u_n\}_{n=1}^\infty \subset \ell_p $ with $ u_n \rightharpoonup u_\alpha^\delta $ as $ n \to \infty $ weakly in $ \ell_\mytau $, we necessarily do have $ \Vert u_n \Vert_p \to \Vert u_\alpha^\delta \Vert_p $, cf.~\cite[Proposition 4.2]{Schusterbuch12} or \cite[Theorem 3.23]{ Scherzerbuch09}. The
Radon--Riesz property of $ \ell_p $, cf.~Proposition~\ref{th:radon_riesz},
then implies
 $ \Vert u_n - u_\alpha^\delta \Vert_p \to 0 $, and thus
 $ \Vert u_n - u_\alpha^\delta \Vert_\mytau \to 0 $ as $ n \to \infty $.

\item[(ii)] In the case $0<p \le 1=\mytau$, for the space $\ell_1$ (with the predual space $c_0$) the weak topology has to be replaced by the weak$^*$-topology. Then the sublevel sets $L_c$ are again bounded subsets of $\ell_r=\ell_1$ for all $c>0$ and due to the Banach--Alaoglu theorem relatively weak$^*$ compact subsets of this space.
This yields the stabilizing property of the penalty here. The weak$^*$-lower semi-continuity of the Tikhonov functional follows from item (c) of Lemma~\ref{lem:basics} by taking into account that
the norm $ \Vert \cdot \Vert_V $ is a weakly lower semi-continuous functional on the Banach space $ V $ and that moreover
$ \Rppur $ is weak$^*$-lower semi-continuous on $ \ell^1 $.
Strong convergence $ \Vert u_n - u_\alpha^\delta \Vert_1 \to 0 $ now follows
similar to the case (i).
For the situations of $0<p<1$ see also the remarks of \cite[p.~385]{Grasmair09}.

\item[(iii)] For $p=0$, we cannot apply \eqref{eq:monoton}, but we can  extend \cite[Lemma~2.1]{Wei13} as follows:
\begin{lemma} \label{lem:Wei}
Let $\{u_n\}_{n=1}^\infty \subset \ell_0$ be a sequence weakly convergent in $\ell_\myr$ to $u_0 \in \ell_\myr$  such there exists a finite upper bound for $\{\|u_n\|_0\}_{n=1}^\infty$.
Then we have $u_0 \in \ell_0$ and $\|u_0\|_0 \le \liminf_{n \to \infty}\|u_n\|_0$.
\end{lemma}
\begin{proof}
After transition to an appropriate subsequence, if necessary and thus without loss of generality, we may assume that
$\|u_n\|_0 = m \in \mathbb{N} $ holds for $ n \ge n_0 $.
If $\|u_0\|_0 \ge m+1 $ would hold, then for some finite set $ K \subset \mathbb{N} $
with
$ \# K = m + 1 $ we have $ u_{0,k} \neq 0 $ for each $ k \in K $. The weak convergence of $ u_n $ then implies that, for some $ n_1 \ge n_0 $, we have $u_{n,k} \neq 0 $ for $ k \in K $ and thus $ \|u_n\|_0 \ge m + 1  $ for  $ n \ge n_1 $. This is a contradiction, thus
$\|u_0\|_0 \le m $ holds.
\end{proof}
Below we show that there exist minimizing sequences $\{u_n\}_{n=1}^\infty$ to the functional $T_\alpha^\delta$ with penalty $\Rp[0]{u}=\|u\|_0$ that are weakly converging to $u_0$ in $\ell_\myr$ and are
bounded in $\ell_0$.
As a consequence of Lemma~\ref{lem:Wei} in combination with the weak-to-weak continuity of $A: \ell_\myr \to V$, we then
have $u_0 \in \ell_0$, and
$u_0$ is a minimizer of $T_\alpha^\delta$.
This is also the case for any other limits of convergent subsequences of $\{u_n\}_{n=1}^\infty$. Now any minimizing sequence $\{u_n\}_{n=1}^\infty$ to $T_\alpha^\delta$ with penalty $\Rp[0]{u}=\|u\|_0$ satisfies for some $K_1>0$ the inequality $\|u_n\|_0 \le K_1$ for all $n \in \N$ and is therefore bounded in $\ell_0$. Moreover, for any $u^\dagger \in \ell_\myr$ with $\|A u^\dagger-v^\delta\|_V \le \delta$ we have some constant $K_2>0$ such that the minimizing sequence  $\{u_n\}_{n=1}^\infty$ to $T_\alpha^\delta$ with penalty $\Rp[0]{u}=\|u\|_0$ satisfies
$\|Au_n-Au^\dagger\|_V^\sigma+ \alpha\,\|u_n\|_0 \le K_2$ (see also for discussions \cite[Section~3]{HKPS07}). Then the left estimate of \eqref{eq:twosided} gives $\|u_n-u^\dagger\|_\myr \le K_2^{1/\sigma}/d_1$ for all $n \in \N$. Hence $\{u_n\}_{n=1}^\infty$ is bounded in $\ell_\myr$ and has therefore some subsequence that is weakly convergent in $\ell_\myr$ to some element $u_0 \in \ell_0$.
For the verification of strong convergence in $ \ell_\mytau $, we consider any minimizing sequence
$ \{ u_n \}_{n=1}^\infty \subset \ell_0 $ converging componentwise to a Tikhonov minimizer $ u \in \ell_0 $. Then
$ \Vert u_n \Vert_0 \to \Vert u_\alpha^\delta \Vert_0 $ (cf.~item (i) for references) and thus
 $ \Vert u_n \Vert_0 = \Vert u_\alpha^\delta \Vert_0 $ for $ n $ large enough. We are thus in a finite-dimensional setting, so $ \Vert u_n - u_\alpha^\delta \Vert_p \to 0 $ as $ n \to \infty $ for any $ p > 0 $.
\end{itemize}
\end{proof}

\section{A priori parameter choice}
\label{apriori_parameter_choice}
In this section, we investigate the regularization properties
of the variational regularization (\ref{eq:Tik}), i.e.~we present convergence rates results for a suitable a priori parameter choice strategy.

As a preparation,
we consider the following non-negative real numbers,
\begin{align}
\label{eq:gamma}
\mygama = \frac{\myr}{\mytau} \cdot \frac{\mytau - q}{\myr-q},
\qquad
\mygamb = \frac{(q-p)\myr}{\myr-q},
\end{align}
where the utilized parameters are introduced in (\ref{eq:parameters}).
The numbers $ \mygama $ and $ \mygamb $ correspond to the convergence rate and the conditional stability rate obtained for our setting as $ \delta \to 0 $,
cf.~the following theorem for the details.
As a further preparation, we introduce the following a priori parameter choice,
\begin{align}
\label{eq:apriori}
\alpha_\delta = \delta^{\sigma+\frac{(q-p)\myr}{\myr-q}}, \quad \delta > 0,
\end{align}
which turns out to be suitable for our setting.
%

We are now in a position to present our main result on the regularization properties of
variational regularization (\ref{eq:Tik}).
\begin{theorem}
\label{th:main}
Let conditions (\ref{eq:parameters})--(\ref{eq:vdelta}) be satisfied.
An a priori parameter choice $ \alpha = \alpha_\delta $
of the form (\ref{eq:apriori})
gives
\begin{align}
\label{eq:main}
\Vert u_{\alpha_\delta}^\delta - u^\dagger \Vert_\mytau &=
\mathcal{O}(\delta^{\mygama}),
\quad
\Rp{u_{\alpha_\delta}^\delta} =
\mathcal{O}(\delta^{-\mygamb})
\quad  \textup{as } \delta \to 0,
\end{align}
where the real numbers $ \mygama $ and $ \mygamb $ are given by (\ref{eq:gamma}).
In addition, $ u_{\alpha}^\delta \in \ell_p $ denotes a minimizer of the Tikhonov functional (\ref{eq:Tik}), and the \searchedsol of equation (\ref{eq:opeq}) satisfies $ u^\dagger \in \ell_q $.
\end{theorem}
The case $ q < p $ means oversmoothing, and the proof utilizes auxiliaries then. Those auxiliaries are generated by hard thresholding; details are given in the following section. The proof of the theorem is then presented in Section \ref{proof_theorem} below.
We proceed with some important comments on the theorem and an example.
\begin{remark} $\;$
\begin{enumerate}
\item We have linear convergence in (\ref{eq:main}) with $ \mygama = 1 $ for $ q = 0 $ when $u^\dagger$ is sparse, and moreover in the limiting case $ \mytau = \myr $ when the problem is well-posed.

\item The convergence rate $ \mygama $ in (\ref{eq:main}) is independent of $ p $, so the particular choice of the stabilizing functional has no impact on the rate of convergence.

\item The second statement in (\ref{eq:main}) is a conditional stability estimate. It is of particular interest for the case $ p = 0 $ and means then
$ \Vert u_{\alpha_\delta}^\delta \Vert_0 = \mathcal{O}(\delta^{-\frac{q \myr}{\myr-q}}) $ as $ \delta \to 0 $. The latter estimate provides information on the
number of non-vanishing entries of the computed approximations.
\end{enumerate}
\end{remark}
\begin{example}
Consider the special case $ r = 2 $, i.e.~we have $ A: \ell_2 \to V $.
Let $ u^\dagger \in \ell_q $, where $ 1 < q < 2 $ is conjugate to the parameter $ a > 2 $, this is, $ \frac{1}{q} + \frac{1}{a}  = 1 $.
We then have $ \gamma_1 = \frac{1}{2} $, i.e.~Theorem~\ref{th:main} yields
$ \Vert u_{\alpha_\delta}^\delta - u^\dagger \Vert_2 = \mathcal{O}(\delta^{1/2}) $ as $ \delta \to 0 $. This convergence rate is reasonable, since the range of the adjoint $ A^\myprime: V^\myprime \to \ell_2 $ of the operator $ A $ satisfies $ \mathcal{R}(A^\myprime) = \ell_q $. The latter identity follows easily from the decomposition
$A=B \circ \mathcal{E}_2^\myr: \ell_2 \to V$,
cf.~Section \ref{spec_model}, in addition with the following two facts: the adjoint operator $ B^\myprime: V \to \ell_q $ is onto since the continuation $ B: \ell_a \to V $ of the operator $ A $ has a bounded inverse on its range $ \mathcal{R}(B) $, and moreover the adjoint
$ (\mathcal{E}_2^\myr)^\myprime: \ell_q \to \ell_2 $
of the embedding operator $\mathcal{E}_2^\myr: \ell_2  \to \ell_\myr $
is again an embedding operator. Note that our convergence result is an improvement over \cite[Proposition 11]{Grasmair_Haltmeier_Scherzer[08]}
with respect to the smoothness assumption on the solution $ u^\dagger $.
 Finally, note that our convergence result holds for any penalty functional $ \Rppur $ with $ 0 \le p \le q $.
\end{example}
\subsection{Hard thresholding}
Below, we consider hard thresholding, which turns out to be an important tool for  Tikhonov regularization in sequence spaces when oversmoothing or post processing is involved. Let
$ \beta > 0 $ be a threshold level.
For any infinite sequence $ u = (u_n)_{n \ge 1 } $, hard thresholding $ H_\beta(u) $ defines an infinite sequence as follows:
\begin{align}
\label{eq:hard_thresh}
H_\beta(u)_n = \left\{ \begin{array}{rl} u_n & \textup{if } \vert u_n \vert \ge \beta, \\
0 & \textup{otherwise},
\end{array} \right. \quad n = 1,2,\ldots \ .
\end{align}
%
%
Hard thresholding satisfies an approximation property as well as an inverse property:
\begin{proposition}
\label{th:hardthreshold}
For any $ 0 \le  p \le q \le \tau $ with $ \tau \ge 1 $, we have
\begin{align*}
\Vert H_\beta(u)- u \Vert_\tau^\tau
\le \Rp[q]{u} \beta^{\tau-q},
\quad
\Rp{H_\beta(u)} \le \Rp[q]{u} \beta^{-(q-p)},
\quad \beta > 0, \  u \in \ell_q.
\end{align*}
\end{proposition}
\begin{proof}
Both estimates are easily obtained:
\begin{align*}
\Vert H_\beta(u)- u \Vert_\tau^\tau
& = \sum_{\vert u_n \vert < \beta} \vert u_n \vert^\tau
\le \Big(\sum_{\vert u_n \vert < \beta} \vert u_n \vert^q\Big) \beta^{\tau-q}
\le \Rp[q]{u} \beta^{\tau-q}, \\
\Rp{H_\beta(u)} &= \sum_{\vert u_n \vert \ge \beta} \vert u_n \vert^p
\le \Big(\sum_{\vert u_n \vert \ge \beta} \vert u_n \vert^q\Big) \beta^{-(q-p)}
\le \Rp[q]{u} \beta^{-(q-p)}.
\end{align*}
Note that the proof applies also to the cases $ p = 0 $ and $ q = 0 $.
\end{proof}
%
%
The two estimates of the proposition may be considered as Jackson- and Bernstein-type inequality, respectively. Similar estimates are considered in \cite[Lemmas 9 and 12]{Miller_Hohage[21]}, \cite[Section 5.1]{Cohen_Dahmen_DeVore[01]} and \cite[Section 7.8]{DeVore[98]}.
\subsection{The basic steps}
\label{basic_steps}
%
In what follows, we shall make frequent use of auxiliary elements, i.e.
\begin{align}
\label{eq:aux}
\uhat_\alpha =  H_{\beta(\alpha)} (u^\dagger), \quad \textup{with }
\beta = \beta(\alpha) := \alpha^{\myr/N}, \quad N := (\myr-q)\sigma + (q-p)\myr,
\end{align}
where the parameter $ \alpha $ corresponds to
variational regularization considered in (\ref{eq:Tik}).
As an immediate consequence of Proposition \ref{th:hardthreshold},
we have the following approximation and stability property,
respectively:
\begin{align}
\label{eq:uhat-estis}
\Vert \uhat_\alpha  - u^\dagger \Vert_\tau^\tau \le \alpha^{(\tau-q)\myr/N} \Rp[q]{u^\dagger},
\quad
\Rp{\uhat_\alpha} \le \alpha^{\kappa-1} \Rp[q]{u^\dagger}
\quad (\alpha > 0),
\end{align}
where
$ 0 \le  p \le q \le \tau $ with $ \tau \ge 1 $, and
the exponent $ 0 < \kappa \le 1 $ is given by
\begin{align}
\label{eq:kappa_def}
\kappa := \frac{\myr -q}{N} \sigma,
\end{align}
with $ N $ taken from (\ref{eq:aux}).
The first estimate in (\ref{eq:uhat-estis}) is applied below both for $ \tau = \myr $ and
$ \tau = r $.

The following proposition provides an upper bound for the minimum of the Tikhonov functional and turns out as a basic ingredient in our analysis.
\begin{proposition}
\label{th:tikfun_esti}
We have
\begin{align}
\label{eq:tikfun_esti}
T_{\alpha}^\delta(u_{\alpha}^\delta) \le c (\alpha^\kappa + \delta^\sigma),
\quad \alpha>0,
\end{align}
where $ c > 0 $ denotes a finite constant that
may depend on the \searchedsolb $ u^\dagger $
but is independent both of $ \alpha $ and $ \delta $.
\end{proposition}
\begin{proof} \textup{
Utilizing auxiliaries from (\ref{eq:aux}) with $ \beta = \beta(\alpha) $, we have
\begin{align*}
T_{\alpha}^\delta(u_{\alpha}^\delta) & \le T_{\alpha}^\delta(\uhat_\alpha) =
\Vert A \uhat_\alpha - v^\delta \Vert_V^\sigma + \alpha \Rp{\uhat_\alpha}
\\
& \le
e_1(\Vert A (\uhat_\alpha - u^\dagger) \Vert_V^\sigma + \delta^\sigma) + \alpha \Rp{ \uhat_\alpha }
\\
& \le e_2(\Vert \uhat_\alpha - u^\dagger \Vert_\myr^\sigma + \delta^\sigma) + \alpha \Rp{\uhat_\alpha} \\
& \le e_3(\alpha^\kappa + \delta^\sigma + \alpha \alpha^{\kappa-1}),
\end{align*}
where the minimization property of $ u_{\alpha}^\delta $ and the
second estimate in (\ref{eq:twosided}) have been applied.
In addition,
both estimates in (\ref{eq:uhat-estis})
have been utilized,
the former with $ \tau = \myr $ in fact.
Moreover, $ e_1, e_2, $ and $ e_3 $ denote finite positive constants
which are independent of both $ \alpha $ and $ \delta $, respectively, and
$ e_3 $ depends on the \searchedsol $ u^\dagger $.
This completes the proof of the lemma.
}
\end{proof}
As an immediate consequence of the preceding proposition, we are able to provide estimates both for
the approximation error in the weaker norm $ \Vert \cdot \Vert_\myr $ and the penalty functional:
\begin{corollary}
\label{th:tikfun_esti_cor}
We have
\begin{align*}
\Vert u_\alpha^\delta - u^\dagger \Vert_\myr
\le
c_1 (\alpha^{(\myr-q)/N} + \delta),
\quad
\Rp{u_\alpha^\delta} \le
c_2 \Big(\alpha^{\kappa-1} + \frac{\delta^\sigma}{\alpha}\Big), \qquad \alpha > 0,
\end{align*}
where $ c_1 $ and $ c_2 $ denote positive constants that are independent both of $ \alpha $ and $ \delta $, respectively.
\end{corollary}
\begin{proof} \textup{
The second estimate of the corollary is immediate from Proposition \ref{th:tikfun_esti},
and the first follows from the first estimate in (\ref{eq:twosided})
and Proposition \ref{th:tikfun_esti}:
\begin{align*}
d_1 \Vert u_\alpha^\delta - u^\dagger \Vert_\myr
\le
\Vert A ( u_\alpha^\delta - u^\dagger) \Vert_V
\le
\Vert A u_\alpha^\delta - \fdelta \Vert_V + \delta
\le e ( \alpha^{\kappa/\sigma} + \delta), \quad \alpha > 0,
\end{align*}
where $ e > 0 $ denotes a finite constant that may depend on the \searchedsol $ u^\dagger $ but is independent of $ \alpha $ and $ \delta $.
}
\end{proof}
\subsection{Proof of Theorem \ref{th:main}}
\label{proof_theorem}
We have already derived estimates of $ \Vert v - u^\dagger \Vert_\myr $ and
$ \Rp{v} $ for both the Tikhonov minimizers $ v =  u_\alpha^\delta $ and the auxiliaries
$ v = \uhat_\alpha $. The proof can now be completed by simple synthesis and by utilizing the interpolation inequality for the scale of standard sequences spaces. We start by a natural decomposition of the approximation error using the auxiliaries:
\begin{align}
\label{eq:proof_thm_1}
\Vert u_\alpha^\delta - u^\dagger \Vert_r
\le \Vert u_\alpha^\delta - \uhat_\alpha \Vert_r
+ \Vert \uhat_\alpha  - u^\dagger \Vert_r.
\end{align}
The first term on the right-hand side in (\ref{eq:proof_thm_1}) can be estimated by
\begin{align}
\label{eq:proof_thm_2}
\Vert \uhat_\alpha  - u^\dagger \Vert_\mytau^\mytau \le e_1 \alpha^{(\mytau-q)\myr/N}, \quad
\alpha > 0,
\end{align}
which follows from the first estimate in (\ref{eq:uhat-estis}) applied with $ \tau = r $.
Here and in what follows, $ e_1, e_2, \ldots $ denote finite positive constants that may depend on the \searchedsolb $ u^\dagger $ but are independent of $ \alpha $ and $ \delta $, respectively.
Below, we estimate the first term on the right-hand side
of (\ref{eq:proof_thm_1}) by utilizing the interpolation inequality in sequence spaces, cf.~(\ref{eq:interpol_ineq}):
%
\begin{align}
\label{eq:interpol_ineq_b}
\Vert u_\alpha^\delta - \uhat_\alpha \Vert_\mytau^\mytau
\le
\Rp{u_\alpha^\delta - \uhat_\alpha}^\theta
\cdot
\Vert u_\alpha^\delta - \uhat_\alpha \Vert_\myr^{\myr(1-\theta)},
\end{align}
with
\begin{align*}
\theta = \frac{\myr-r}{\myr-p}, \quad 1-\theta = \frac{r-p}{\myr-p}.
\end{align*}
The first term on the right-hand side of (\ref{eq:interpol_ineq_b}) can be estimated by means of the conditional stability estimate in (\ref{eq:uhat-estis}) and Corollary
\ref{th:tikfun_esti_cor}:
\begin{align}
\label{eq:proof_thm_3}
\Rp{u_\alpha^\delta - \uhat_\alpha}
\le e_2
(\Rp{u_\alpha^\delta}  + \Rp{\uhat_\alpha})
\le e_3
\Big(\alpha^{\kappa-1} + \frac{\delta^\sigma}{\alpha}\Big), \quad \alpha > 0.
\end{align}
The second term on the right-hand side of (\ref{eq:interpol_ineq_b}) can also be treated by means of the estimates in (\ref{eq:uhat-estis}) and Corollary
\ref{th:tikfun_esti_cor}:
%
\begin{align}
\label{eq:proof_thm_4}
\Vert u_\alpha^\delta - \uhat_\alpha \Vert_\myr
\le
\Vert u_\alpha^\delta - u^\dagger \Vert_\myr
+ \Vert \uhat_\alpha  - u^\dagger \Vert_\myr
\le
e_4 (\alpha^{(\myr-q)/N} + \delta), \quad \alpha > 0.
\end{align}
We now consider the above estimates for the specific a priori parameter choice
$ \alpha=\alpha_\delta $ introduced in (\ref{eq:apriori}).
In fact, this parameter choice
balances the two terms of the upper bounds in
(\ref{eq:proof_thm_3}) and
(\ref{eq:proof_thm_4}), respectively. This yields
the two estimates
\begin{align}
\label{eq:proof_thm_5}
\Rp{ u_{\alpha_\delta}^\delta - \uhat_{\alpha_\delta} }
\le
e_5 \delta^{-\mygamb},
\quad \Vert u_{\alpha_\delta}^\delta - \uhat_{\alpha_\delta} \Vert_\myr
\le e_6 \delta,
\end{align}
and in addition we obtain the conditional stability estimate of the theorem, cf.~the second estimate in (\ref{eq:proof_thm_3}). Utilizing in (\ref{eq:interpol_ineq_b})
the two estimates from (\ref{eq:proof_thm_5}) finally gives
%
\begin{align*}
\Vert u_{\alpha_\delta}^\delta -  \uhat_{\alpha_\delta}  \Vert_\mytau \le
e_7\delta^{\mygama}.
\end{align*}
The tedious but simple computations are left to the reader.
For the parameter choice (\ref{eq:apriori}), in addition we also have
$ \Vert \uhat_{\alpha_\delta} - u^\dagger  \Vert_\mytau \le
e_8\delta^{\mygama} $, cf.~(\ref{eq:proof_thm_2}).
This completes the proof of Theorem~\ref{th:main}.
%
%
\begin{remark}
In the case $ q = p $, i.e.~$ \Rppur $ is the stabilizing functional and $ u^\dagger \in \ell_p $, there is no oversmoothing. The proof technique can be simplified in that case, and auxiliaries are not needed then in fact.
\end{remark}
\subsection{Post processing}
Under the conditions of Theorem \ref{th:main}, we can guarantee sparsity only for
$ p = 0 $, i.e.~the penalty functional is $ \Rppur[0] $. In all other cases, sparsity can easily be obtained by some simple post processing based on hard thresholding.
For this purpose, consider
\begin{align}
\label{eq:post_proc}
v_\beta^\delta = H_\beta(u_{\alpha_\delta}^\delta), \quad \beta > 0,
\end{align}
where $ H_\beta $ refers to hard thresholding given by (\ref{eq:hard_thresh}).
For an appropriate choice of $ \beta = \beta_\delta $, the element $ v_{\beta_\delta}^\delta $ allows the same error rates as the original regularizing element
$ u_{\alpha_\delta}^\delta $, but this time with guaranteed sparsity for any $ p $ and with available conditional stability estimate, as the following corollary shows.
%
%
\begin{corollary}[Post processing]
\label{th:postproc}
Let the conditions of Theorem \ref{th:main} hold, and in addition let
$ v_\beta^\delta $ be as in
(\ref{eq:post_proc}) and
$ \beta = \beta_\delta = \delta^{\frac{\myr}{\myr-q}}, \, \delta > 0 $.
Then we have
%
%
\begin{align}
\label{eq:postproc}
\Vert v_{\beta_\delta}^\delta - u^\dagger \Vert_\mytau &=
\mathcal{O}(\delta^{\mygama}),
\quad
\Rp{ v_{\beta_\delta}^\delta } =
\mathcal{O}(\delta^{-\mygamb})
\quad \textup{ as } \delta \to 0,
\end{align}
where $ \mygama, \mygamb $ are given by (\ref{eq:gamma}).
\end{corollary}
\begin{proof}
We first consider the left estimate in (\ref{eq:postproc}). Due to the first estimate in Theorem~\ref{th:main}, it is sufficient to provide a suitable estimate for $ \Vert v_{\beta_\delta}^\delta - u_{\alpha_\delta}^\delta \Vert_\mytau $. For this, we make use of the first estimate in Proposition~\ref{th:hardthreshold}, applied with $ u = u_{\alpha_\delta}^\delta, \tau = \mytau $, and $ q = p $. In addition, we utilize the second estimate in Theorem \ref{th:main}:
\begin{align*}
\Vert v_{\beta_\delta}^\delta - u_{\alpha_\delta}^\delta \Vert_\mytau^\mytau
& \le
\beta_\delta^{\mytau- p} \Rp{u_{\alpha_\delta}^\delta } =
\delta^{\frac{\mytau- p}{\myr-q}\myr} \mathcal{O}(\delta^{-\mygamb}) =
\mathcal{O}(\delta^{\mygama \mytau}).
\end{align*}
%
Next we consider the second estimate in (\ref{eq:postproc}). It immediately follows from
the second estimate in Theorem~\ref{th:main} and the second estimate in Proposition~\ref{th:hardthreshold}, applied with
$ q = p $:
\begin{align*}
\Rp{ v_{\beta_\delta}^\delta } =
\Rp{H_{\beta_\delta}(u_{\alpha_\delta}^\delta)}
\le \Rp{u_{\alpha_\delta}^\delta} = \mathcal{O}(\delta^{-\mygamb}),
 \quad \delta > 0.
\end{align*}
This completes the proof.
\end{proof}
\section{Conclusions and outlook}
\label{conclusions_outlook}
In this paper, we consider variational regularization of linear ill-posed problems in sequence spaces. Our focus lies both on well-posedness of the considered Tikhonov functional and convergence rates for suitable a priori parameter choices.
The considerations are restricted to operators that satisfy a stability estimate with respect to a weaker sequence space. On the other hand, we allow an oversmoothing of the penalty functional. Sparsity results are obtained either when the $ \ell_0 $-norm is utilized as penalty functional or by post-processing based on hard thresholding.


In what follows, we list some topics that are beyond the scope of this paper but could be the subject of further investigations.
First, the analysis of a posteriori parameter choice strategies
 like the discrepancy principle will be of interest. Note that those strategies are superior to a priori parameter choice strategies, since they do not require any knowledge of the smoothness of the solution.
\bh{Another topic of strong interest} will be the investigation of problems with solutions satisfying no additional smoothness conditions, i.e.~they belong to the approximation error space $ \ell_\mytau $ but not to any stronger sequence space $ \ell_q, \, q < \mytau $.
A discussion of the optimality of the presented rates \bh{can certainly play some role}.
Moreover, it seems that the mathematical approach presented in this paper can also be used for the treatment of nonlinear problems.
\bh{The analysis of sparsity-promoting features of
variational regularization with penalty functionals that are related
to an $ \ell_p $-norm, with $ 0 < p \le 1 $ may also be exciting}.

Last but not least: extensions of the presented result to weighted and weak $ \ell_p $ spaces seem to be possible. The former case may allow the coverage of other applications, like period integral operators in spaces of periodic Sobolev spaces.

%
\bibliography{PH25}
\end{document}